\documentclass[a4paper,11pt]{amsart} 
\usepackage{amssymb, amsmath}
\usepackage{amscd}
\usepackage{amsthm}
\numberwithin{equation}{section}
\usepackage{epsfig}
\usepackage{amsmath}
\usepackage{amsfonts,amssymb,amsopn}

\usepackage{amsthm}
\usepackage{hyperref}
\usepackage{verbatim}
\usepackage{color}
\usepackage[color,all]{xy}

\newcommand{\cI}{{\mathcal I}}
\newcommand{\cL}{{\mathcal L}}
\newcommand{\cO}{{\mathcal O}}
\newcommand{\C}{{\mathbb C}}
\newcommand{\PP}{{\mathbb P}}
\newcommand{\bZ}{{\mathbb Z}}

\newcommand{\tM}{\widetilde{M}}
\newcommand{\tN}{\widetilde{N}}

\newcommand{\Ker}{\mathrm{Ker}}

\newcommand{\Image}{\mathrm{Im}\,}
\newcommand{\Iden}{\mathrm{Id}}
\newcommand{\Hom}{\mathrm{Hom}}
\newcommand{\rank}{\mathrm{rk}\,}
\newcommand{\Pic}{\mathrm{Pic}}

\newcommand{\HN}{\mathrm{HN}}
\newcommand{\ev}{\mathrm{ev}}
\newcommand{\isom}{\xrightarrow{\sim}}

\newcommand{\Kc}{K_{C}}
\newcommand{\Oc}{{{\cO}_C}}

\newcommand{\Oz}{{{\cO}_Z}}

\newcommand{\Span}{\mathrm{Span}\,}

\newcommand{\Spec}{\mathrm{Spec}\,}
\newcommand{\Quot}{\mathrm{Quot}}
\newcommand{\Hilb}{\mathrm{Hilb}}

\newcommand{\Ann}{\mathrm{Ann}}
\newcommand{\cl}{{\mathrm{cl}}}

\newcommand{\Gr}{\mathrm{Gr}}

\newcommand{\mult}{\mathrm{mult}}
\newcommand{\length}{\mathrm{length} \,}

\newcommand{\gen}{{\mathrm{gen}}}

\newcommand{\mev}{M_{E , V}}

\newcommand{\mevd}{M^\vee_{E , V}}
\newcommand{\mfw}{M_{F , W}}
\newcommand{\mfwd}{M^\vee_{F , W}}
\newcommand{\reddeg}{\mathrm{red\,deg\,}}

\newcommand{\bG}{\overline{G}}
\newcommand{\tG}{\widetilde{G}}

\newcommand{\opo}{\cO_{\PP^1}}
\newcommand{\pe}{\PP E^\vee}
\newcommand{\pf}{\PP F^\vee}
\newcommand{\pv}{{\PP V^\vee}}
\newcommand{\pw}{{\PP W^\vee}}
\newcommand{\ope}{\cO_{\pe}}
\newcommand{\opeo}{\ope (1)}
\newcommand{\opfo}{\cO_{\PP F^\vee} (1)}
\newcommand{\Iz}{\cI_Z}

\newtheorem{theorem}{{\textbf Theorem}}[section]
\newtheorem*{theorem*}{{\textbf Theorem}}
\newtheorem{proposition}[theorem]{{\textbf Proposition}}
\newtheorem{corollary}[theorem]{{\textbf Corollary}}
\newtheorem{lemma}[theorem]{{\textbf Lemma}}
\newtheorem{conjecture}[theorem]{{\textbf Conjecture}}
\newtheorem{remit}[theorem]{{\textbf Remark}}
\newenvironment{remark}{\begin{remit}\rm}{\end{remit}}
\newtheorem{question}[theorem]{{\textbf Question}}
\newtheorem{exait}[theorem]{{\textbf Example}}
\newenvironment{example}{\begin{exait}\rm}{\end{exait}}
\newtheorem{defnit}[theorem]{\textbf{Definition}}
\newenvironment{definition}{\begin{defnit}\rm}{\end{defnit}}

\title[Geometry of linearly stable coherent systems]{Geometry of linearly stable coherent systems over curves}

\author{Abel Castorena}
\email{abel@matmor.unam.mx}
\address{Centro de Ciencias Matem\'aticas -- UNAM Campus Morelia, Antigua Carretera a P\'atzcuaro \# 8701, Col.\ Ex Hacienda San Jos\'e de la Huerta, Morelia, Michoac\'an, Mexico C.\ P.\ 58089.}

\author{George H.\ Hitching}
\email{gehahi@oslomet.no}
\address{Oslo Metropolitan University, Postboks 4, St. Olavs plass, 0130 Oslo, Norway.}

\keywords{Linear stability; generated coherent system; scroll}

\subjclass{Primary 14H60; Secondary 14H51.}

\begin{document}

\begin{abstract}
Let $E$ be a vector bundle over a smooth curve $C$, and $V$ a generating space of sections of $E$. We characterise Mumford linear stability of the associated projective model of $\pe$ in $\pv$ in terms of geometric and cohomological properties of the coherent system $(E, V)$, and give some applications. We show that any $\PP^{r-1}$-bundle over $C$ has a linearly stable model in $\PP^{n-1}$ for any $n \ge r+2$. Furthermore; linear stability of $(E, V)$ is a necessary condition for stability of the kernel bundle of $(E, V)$, which is predicted by Butler's conjecture for general $C$ and $(E, V)$. We give new examples showing that it is not in general sufficient; in particular, a general bundle $E$ of large degree fits into a linearly stable coherent system $(E, V)$ with nonsemistable kernel bundle. Finally, we use these ideas to show the stability of $\mev$ for certain $(E, V)$ of type $(r, d, r+2)$ where $E$ is not necessarily stable.
\end{abstract}

\maketitle

\section{Introduction}

Let $X \subset \PP^{n-1}$ be a complex projective variety of dimension $r$. In the context of geometric invariant theory, D.\ Mumford defined \textit{linear semistability} of $X$ as follows. The \textsl{reduced degree} of $X$ in $\pv$ is defined as
\[ \reddeg X \ := \ \frac{\deg X}{n - r} . \]
Moreover, if $W \subseteq V$ is a vector subspace of dimension $m$, we write $p_W$ for the projection $\pv \dashrightarrow \pw$ with centre $\PP W^\perp = \PP^{n - m - 1}$. The following is \cite[Definition 2.16]{Mum}.

\begin{definition} \label{MumfordLinSst}
Let $X \subset \PP^{n-1}$ be as above. Then $X$ is said to be \textsl{(Mumford) linearly semistable} in $\PP^{n-1}$ if for all subspaces $W \subset V$ such that $p_W (X) \subseteq \pw$ has dimension $r$, we have $\reddeg X \le \delta \cdot \reddeg p_W ( X )$, where $\delta = \deg p_W|_X$.
 If inequality is strict for all such $W$, then $X$ is said to be \textsl{linearly stable}.
\end{definition}

Intuitively, linear stability implies that $X$ should not have many secants with large defect. For example, by \cite[Proposition 2.4]{BT}, if $X$ is linearly stable then it has no points of multiplicity greater than $\frac{\deg X}{n - r}$.

Mumford showed in \cite[Theorem 4.12]{Mum} that the Chow form of a linearly stable curve is stable in the GIT sense, and obtained a result in the same direction for varieties of higher dimension in \cite[Proposition 2.17]{Mum}. The statement for curves has found application in \cite{BT} and elsewhere; see \cite{CHL} for discussion. In particular, as was made explicit in \cite{MS}, linear stability of curves has implications for the well known conjecture of Butler \cite{But97}. This line of inquiry is extended to higher rank in \cite{CHL}, \cite{CH} and \cite{BO}, and will be further discussed below.

The goal of the present work is to continue the study of linear stability in the case where $X$ is the image of a projective bundle $\pe$ over a curve $C$, and $V$ a generating subspace of $H^0 ( \pe, \opeo ) \cong H^0 (C, E)$. To a greater extent than in \cite{CHL} and \cite{CH}, we focus on geometric aspects. The situation is closely related to the study of coherent systems over $C$, whose definition we now recall.

\begin{definition} \label{CohSysDefn}
A \textsl{coherent system of type $(r, d, n)$ over $C$} is a pair $(E, V)$ where $E \to C$ is a vector bundle of rank $r$ and degree $d$, and $V \subseteq H^0 (C, E)$ is a subspace of dimension $n$. Such an $(E, V)$ is \textsl{generated} if the evaluation homomorphism $\ev_{E, V} \colon \Oc \otimes V \to E$ a is surjective vector bundle map.
\end{definition}

The content of the article is as follows. In {\S} \ref{Prelims}, we recall or prove some preliminary statements on scrolls in projective space and generated coherent systems. Next, we prove a key technical result, Proposition \ref{LinSstCrit}, which characterises linear stability of a scroll $\pe \to \pv$ in terms of defective secants and in terms of the cohomology of subsystems of $(E, V)$. This has several applications.

Firstly, we recall that a notion of linear semistability for a generated coherent system $(E, V)$ over $C$ was proposed in \cite[Definition 5.1]{CHL}. It was claimed in \cite[Remark 5.18]{CHL} that this property is equivalent to linear semistability of the image of $\pe \to \pv$ in the sense of Definition \ref{MumfordLinSst}. It turns out that this statement is somewhat misleading: Although the respective notions of linear \emph{semistability} agree, Mumford linear \emph{stability} is slightly weaker when the bundle $E$ is not ample. 
 Therefore, a correspondingly weaker definition of linear (semi)stability for coherent systems was adopted in \cite{CH} (Definition \ref{DefnLinSstCohSys} below). We prove in Theorem \ref{EquivDefns} that this agrees with Mumford linear (semi)stability of the image of $\pe \to \pv$ for all generated $(E, V)$. This further justifies the use of the term ``linear stability'' with coherent systems.

Another application of Proposition \ref{LinSstCrit} is to prove the following.

\begin{theorem*}[Theorem \ref{AnyELinSt}]
Let $C$ be any smooth curve and $E \to C$ any vector bundle of rank $r$. Then for any $n \ge r+2$, there exists a projective model $\varphi \colon \PP E^\vee \to \PP^{n-1}$ which is linearly stable.
\end{theorem*}

\noindent The strategy is to embed $\pe$ sufficiently very amply in a projective space of large dimension and then take a general projection. In view of the relevance of linear stability for GIT and moduli, it is interesting that the degree of the embedding required turns out to depend on the slope of a maximal subbundle of $E^\vee$. This will be a subject of further inquiry; see also Question \ref{QuestionLinStChow} below.

In {\S\S} \ref{Counterexamples} and \ref{rdr+2}, we apply these results to the question of stability of \textsl{kernel bundles}. The kernel bundle $\mev$ of a generated coherent system $(E, V)$ is defined by the exact sequence
\begin{equation} \label{DSBseq}
0 \ \to \ \mev \ \to \ \Oc \otimes V \ \to \ E \ \to \ 0 .
\end{equation}
Butler's conjecture \cite[Conjecture 2]{But97}, which will be discussed further in {\S} \ref{RecallButler}, would imply in particular that $\mev$ is stable for a suitably general choice of $C$ and generated $(E, V)$. By the discussion in \cite[{\S} 1.3]{CH} (see also \cite[Remark 3.2]{MS}), if $\mev$ is a stable vector bundle then $(E, V)$ is a linearly stable coherent system. The converse is true under certain circumstances, as shown for example in \cite[Theorem 5.1]{MS} and \cite[Proposition 5.10]{CHL}, and implicitly in \cite[Lemma 2.2]{Mis}. There are, however, several counterexamples showing that without further assumptions, slope stability of $\mev$ is stronger than linear stability of $(E, V)$: see \cite[{\S} 8]{MS}, \cite[Theorem 4.1]{CT} and \cite[Theorem 4.2]{CHL} for rank one coherent systems, and \cite[Theorem 5.17]{CHL} for type $(2, d, 4)$. We give two new classes of such counterexamples.

Firstly, using pullbacks from $\PP^1$ as in \cite[{\S} 8]{MS}, for $n \ge r + 2$ and $C$ arbitrary, we use Theorem \ref{AnyELinSt} to construct linearly stable systems $(E, V)$ of type $(r, d, n)$ over $C$ with $\mev$ not semistable (Example \ref{PullbackCounterex}). This generalises \cite[Theorem 5.17]{CHL}.

Now the lack of semistability in the above-mentioned example comes from the fact that $\mev$ (and $E$) are pullbacks from $\PP^1$, and therefore sums of line bundles by Grothendieck's theorem. Regarding more general (for example, stable) $E$, we have the following.

\begin{theorem*}[Theorem \ref{AnyENotSst}]
Let $C$ be any curve of genus $g \ge 2$. For any $r \ge 2$, there exists an integer $d_0 (r)$ such that if $E$ is any stable bundle of rank $r$ and degree $d \ge d_0 (r)$, then there exists a generating subspace $V \subseteq H^0 ( C, E )$ of dimension $r + 2$ such that $(E, V)$ is generated and linearly stable, but the rank two bundle $\mev$ is not semistable.
\end{theorem*}

\noindent In particular, for $E$ of large degree, linear stability of $(E, V)$ never implies semistability of $\mev$ without some additional hypothesis on $(E, V)$.

We conclude with an example in the opposite direction: In {\S} \ref{rdr+2}, we construct linearly stable systems of type $(r, d, r+2)$ where $d$ is small, and use a generalisation of \cite[Proposition 5.10]{CHL} to show that $\mev$ is stable. We point out that $\left( E, H^0 (C, E) \right)$, although ample and generated, is not necessarily $\alpha$-stable (Definition \ref{alphaSt}); one motivation for the present investigation 
 is to understand the relation between linear stability and $\alpha$-stability of the coherent system $(E, V)$ on the one hand, and slope stability of $\mev$ on the other. This continues the line of inquiry of \cite{CH}, where coherent systems are constructed exhibiting various combinations of linear stability and GIT-stability.

Regarding topics of future investigation: Another reason for studying geometric aspects of linearly stable scrolls is that, as mentioned above, a linearly stable curve is Chow stable. Now scrolls are unusual among higher dimensional subvarieties of $\PP^{n-1}$ in that their secant and osculating spaces behave in many respects like those of curves. Therefore, the following was posed in \cite[Question 6]{CHL}.

\begin{question} \label{QuestionLinStChow}
Suppose that $\psi \colon \pe \to \pv$ is a linearly (semi)stable scroll. Is $\pe$ Chow (semi)stable?
\end{question}

A positive answer to Question \ref{QuestionLinStChow} might open interesting possibilities for studying moduli of scrolls, in the same way as for curves in \cite{Mum} and \cite{BT}. We hope that the results of this paper may be of help in answering this question.

Another problem of interest is to give a reasonable definition of linear stability for generated coherent systems $(E, V)$ over a higher dimensional variety $X$. Here one must take account of the choice of polarisation on $X$, and the ampleness properties of $(E, V)$ also play a significant role. We propose to investigate this question in the future.

\subsection*{Acknowledgements}

We thank Ali Bajravani and Ragni Piene for helpful discussions. The first author is supported by grant PAPIIT IN100723 ``Curvas, sistemas lineales en superficies proyectivas y fibrados vectoriales'' from DGAPA, UNAM. The second author thanks the Centro de Ciencias Matem\'aticas of UNAM Morelia for financial support, hospitality and excellent working conditions during visits in 2023 and 2024.

\section{Preliminaries} \label{Prelims}

Throughout this work, $C$ denotes a projective smooth curve of genus $g \ge 0$ over $\C$. In this section, we study scrolls over $C$, their secants, and their relation to generated coherent systems. We also give a brief overview of Butler's conjecture.

\subsection{Scrolls}

In the literature, a ``scroll over $C$'' is most often taken to be a projective bundle over $C$ equipped with an embedding in projective space which is linear on fibres. However, as linearly stable varieties may be singular, and for other reasons, it will be convenient for us to relax the condition that the linear system be ample.

\begin{definition} \label{DefnScroll}
We define a \textsl{scroll} as the image of the natural map $\pe \to \pv$ where $E$ is a vector bundle over $C$ and $V$ a generating subspace of $H^0 (\pe , \opeo)$. 
\end{definition}

\subsubsection{Maps to projective space}

Let $E \to C$ be a vector bundle of rank $r$ and degree $d > 0$. Let $W \subseteq H^0 (\pe , \opeo)$ be a subspace of dimension $n$. Via the natural identification
\[
H^0 ( \pe , \opeo ) \ \isom \ H^0 (C, E) ,
\]
we can identify $W$ with a subspace of $H^0 (C, E)$. Then the natural map $\psi \colon \pe \dashrightarrow \PP W^\vee$ can be realised by dualising the evaluation map $\ev \colon \Oc \otimes W \to E$, projectivising, and projecting to $\pw$ as follows:
\begin{equation} \label{DefnPsi}
\pe \ \dashrightarrow \ \pw \times C \ \to \ \pw .
\end{equation}

\begin{lemma} \label{DegreeProperties} Let $C$, $E$ and $\psi \colon \pe \dashrightarrow \pw$ be as above.
\begin{enumerate}
\renewcommand{\labelenumi}{(\alph{enumi})}
\item The series $\PP W \subseteq |\opeo|$ is base point free if and only if $W$ generates $E$.
\item Write $E_W$ for the subsheaf of $E$ generated by $W$. Then $\psi \colon \pe \dashrightarrow \pw$ factorises via a \emph{morphism} $\PP E_W^\vee \to \pw$.
\item Suppose that $W$ generates $E$. Then $( \deg \psi ) \cdot \deg \psi ( \pe ) = d$.
\end{enumerate}
\end{lemma}

\begin{proof}
Statement (a) follows from the definition (\ref{DefnPsi}). For (b): We observe that the evaluation map factorises as $\Oc \otimes W \to E_W \to E$ where the first map is a vector bundle surjection and the second is a sheaf injection. Dualising, projectivising and projecting as in (\ref{DefnPsi}), we obtain the desired factorisation. 

As for (c): As $d > 0$, we have $\dim \, \psi ( \pe ) = r$. Let $\Pi \subseteq W$ be a subspace of dimension $r$ which is general in the sense that 
\[
\pe \times_{\pv} \PP \Pi^\perp
\]
is a degree zero subscheme of length equal to $(\deg \psi ) \cdot \deg \psi (\pe)$. Transposing the evaluation map $\ev \colon \Oc \otimes \Pi \to E$, we obtain
\begin{equation} \label{WdElemTrans}
0 \ \to \ E^\vee \ \xrightarrow{^t\ev} \ \Oc \otimes \Pi^\vee \ \to \ \tau \ \to \ 0 .
\end{equation}
We deduce that there is a diagram of varieties
\[ \xymatrix{ & \pe \ar[d] \ar@{-->}[dr] & \\
 \PP \Pi^\perp \ar@{^{(}->}[r] & \pw \ar@{-->>}[r] & \PP \Pi^\vee . } \]
Furthermore, as $W$ generates $E$, after deforming $\Pi$ if necessary, by \cite[Lemma 5.4]{Hit20} we may assume that $\tau$ has reduced support. Then $\pe \times_{\pv} \PP \Pi^\perp$ is identified with the projectivisation of the set $\{ e \in E : {^t\ev} (e) = 0 \}$. From (\ref{WdElemTrans}) we see that this is a set of $\deg E$ reduced points, as desired.
\end{proof}

\subsubsection{Subschemes of dimension zero and elementary transformations} \label{Zet}

Here we recall some facts from \cite{Hit20} and \cite{Hit23}. Let $\psi \colon \pe \to \pv$ be a scroll, and let $Z \subset \pe$ be a closed subscheme. Tensoring the ideal sheaf sequence of $Z$ by $\opeo$ and taking direct image by $\pi \colon \pe \to C$, we obtain
\begin{equation} \label{EZseq}
0 \ \to \ \pi_* \Iz (1) \ \to \ E \ \to \ \pi_* \cO_Z (1) \ \to \ \cdots
\end{equation}
We will write $E^Z$ for the sheaf $\pi_* \Iz (1)$. If $Z$ is of dimension zero, then it is supported over finitely many points of $C$ and $E^Z$ is an elementary transformation of $E$; that is, a full rank subsheaf.

The following is essentially \cite[Lemma 2.14]{Hit23}, which is a generalisation of \cite[Propositions 4.2 and 4.4]{Hit20}. For the reader's convenience, we give a proof here.

\begin{proposition} \label{SpanZ}
Suppose that $V \subseteq H^0 (C, E)$ is a generating subspace of dimension $n$, and let $\psi \colon \pe \to \pv$ be the associated map to $\PP^{n - 1}$.
\begin{enumerate}
\item[(a)] For each closed subscheme $Z \subset \pe$, we have
\[
\Span \psi (Z) = \Ker \left( V^\vee \to \left( V \cap H^0 \left( C, E^Z \right) \right)^\vee \right) .
\]
\item[(b)] Suppose that $E^{Z'} = E^Z$ as subsheaves of $E$; that is, $E^{Z'}$ and $E^Z$ define the same element of the appropriate Quot scheme of $E$. Then $\Span \psi (Z) = \Span \psi (Z')$.
\end{enumerate}
\end{proposition}

\begin{proof}
(a) By definition, $\Span \psi (Z)$ is the intersection of all codimension one subspaces of $\pv$ containing $\psi (Z)$; that is, hyperplanes in $\pv$ defined by sections belonging to $V \cap H^0 ( \pe , \Iz (1) )$. By linear algebra, this coincides with
\begin{equation} \label{SpanAsKernel}
\PP \Ker \left( V^\vee \ \to \ (V \cap H^0 ( \pe , \Iz (1) )^\vee \right) .
\end{equation}
As $E^Z = \pi_* \Iz (1)$, via the canonical identification $H^0 (C, E) \isom H^0 (\pe, \opeo)$ we have an identification
\[
V \cap H^0 \left( C, E^Z \right) \ \isom \ V \cap H^0 (\pe, \Iz (1)) .
\]
Part (a) now follows from (\ref{SpanAsKernel}). For the rest: If $E^{Z'} = E^Z$ as subsheaves of $E$, then $V \cap H^0 \left( C, E^{Z'} \right) = V \cap H^0 \left( C, E^Z \right)$. Now (b) follows from (a).
\end{proof}

Now subschemes $Z$ with \textsl{defect}; that is, with $\Span \psi (Z)$ of smaller dimension than expected, are fundamental to the study of linear stability. There is, however, one type of defective secant which we will wish to avoid; those called \textsl{$\pi$-defective} in \cite{Hit20}:

\begin{definition} \label{DefnPiNondef}
Let $\pi \colon \pe \to C$ be a projective bundle. A length $e$ subscheme $Z \subset \pe$ is said to be \textsl{$\pi$-nondefective} if the map $E \to \pi_* \cO_Z (1)$ in (\ref{EZseq}) is surjective; equivalently, if $\deg E^Z = \deg E - e$. Otherwise, we say that $Z$ is \textsl{$\pi$-defective}.
\end{definition}

If $Z$ is $\pi$-defective, then $\psi (Z)$ has defect in $\pv$. However, this arises not from the global geometry of $\pe$ in $\pv$, but from the fact that any subscheme contained in a fibre of $\pe$ can span at most a $\PP^{r-1}$ in $\pv$. See \cite[{\S} 3.1]{Hit20} and \cite[{\S} 2.5.1]{Hit23} for further discussion. 

Recall now that a zero-dimensional subscheme is said to be \textsl{curvilinear} if it is contained in a smooth curve; equivalently, if its tangent spaces have dimension at most one.

\begin{theorem} \label{F=EZ}
Let $E \to C$ be a vector bundle and $\pi \colon \pe \to C$ the associated projective bundle. Then for any $e \ge 1$ and any elementary transformation $[F \subset E] \in \Quot^e (E)$, there exists a $\pi$-nondefective and curvilinear $Z \in \Hilb^e \left( \pe \right)$ such that $F = E^Z$. 
\end{theorem}

\begin{proof}
The existence of $Z$ is \cite[Theorem 2.2 (a)]{Hit20}, and curvilinearity is shown in the proof of \cite[Theorem 2.6 (a)]{Hit20}. 
\end{proof}

\begin{remark}
By \cite[Theorem 2.2 (b)]{Hit20}, if $[F \subset E]$ is general in $\Quot^e (E)$ in the sense that $E/F$ has reduced support on $C$, then the $Z$ of the theorem is unique. More precisely; the choice of such an $F$ is canonically equivalent to a choice of $e$ elements $\phi_1 , \ldots , \phi_e$ of $\pe$ lying in distinct fibres of $\pi \colon \pe \to C$, and then $Z = \{ \phi_1 , \ldots , \phi_e \}$.
\end{remark}

\subsection{Butler's conjecture} \label{RecallButler}

We now recall a well known conjecture of Butler, which is of fundamental relevance to the study of linear stability.

Although our focus is on linear semistability, by far the most studied notion of semistability for coherent systems is the following notion of \textsl{$\alpha$-semistability} for a real parameter $\alpha$, which arises directly from GIT.

\begin{definition} \label{alphaSt}
Let $\alpha$ be a positive real number. A coherent system $(E, V)$ (not necessarily generated) is said to be \textsl{$\alpha$-semistable} if for all proper subbundles $F \subset E$ and all subspaces $W \subseteq V \cap H^0 (C, F)$, we have 
\begin{equation} \label{alphaSlopeIneq}
\frac{\deg F}{\rank F} + \alpha \cdot \frac{\dim W}{\rank F} \ \le \ \frac{\deg E}{\rank E} + \alpha \cdot \frac{\dim V}{\rank E} .
\end{equation}
We say that $(E, V)$ is \textsl{$\alpha$-stable} if inequality is strict for all such pairs $(F, W)$.
\end{definition}

King and Newstead \cite{KN} constructed moduli spaces $G(r, d, n; \alpha)$ for $\alpha$-semistable coherent systems. The literature on moduli of coherent systems is very considerable; we refer the reader to \cite{New} for a recent survey. A detailed comparison of $\alpha$-stability and linear stability can be found in \cite{CH}.

For a generated coherent system $(E, V)$, as mentioned in the introduction, the kernel bundle $\mev$ is defined by the exact sequence
\[
0 \ \to \ \mev \ \to \ \Oc \otimes V \ \to \ E \ \to \ 0 .
\]
The bundle $\mev^\vee$ is called the \textsl{dual span bundle} (DSB) of the pair $(E, V)$, and $\left( \mev^\vee, V^\vee \right)$ is again a generated coherent system. Butler's conjecture \cite[p.\ 4]{But97} may be stated as follows.

\begin{conjecture}[Butler]
Let $C$ be a general curve. Suppose that $n > r$ and $\frac{1}{\alpha} \gg 0$. Then a general element of $G(r, d, n; \alpha)$ is generated, and the association $(E, V) \mapsto \left( \mev^\vee, V^\vee \right)$ defines a birational equivalence $G(r, d, n; \alpha) \ \dashrightarrow \ G(n-r, d, n; \alpha)$.
\end{conjecture}

\noindent By the discussion on \cite[p.\ 3]{But97} (see also \cite[Lemma 3.1]{CH}), the conjecture would imply in particular that for a general element $(E, V)$ of each component of $G(r, d, n; \alpha)$, the vector bundle $\mev$ is semistable.

Butler's conjecture has been proven in many cases; see for example \cite{BBN} for $r = 1$, and \cite{BMNO} and \cite[Theorem 6.16]{BO} for some cases with $r \ge 2$. However, it remains open for most types $(r, d, n)$ with $r \ge 2$. In {\S} \ref{rdr+2}, we will use linear stability to prove the slope stability of $\mev$ for certain $(E, V)$ of type $(r, d, r+2)$.

\subsection{Structure of generated coherent systems}

In this subsection we prove some facts we will need in the study of linear stability of generated coherent systems.

The following lemma must be known, but we were unable to locate it in the literature. Recall that a vector bundle $E$ (over any variety) is said to be \textsl{ample} if $\opeo$ is ample on $\pe$. By \cite[Theorem 2.4]{Har71}, a bundle $E$ over a complex projective curve $C$ is ample if and only if all quotients of $E$ have positive degree.

\begin{lemma} \label{StructureGeneratedCohSys}
Let $(E, V)$ be a generated coherent system over $C$.
\begin{enumerate}
\item[(a)] The vector bundle map $\ev_{E^\vee}$ is everywhere injective. In particular, $h^0 (C, E^\vee) \le r$.
\item[(b)] There is a split exact vector bundle sequence
\[
0 \ \to \ E_1 \ \to \ E \ \xrightarrow{\left( \ev_{E^\vee} \right)^\vee} \ \Oc \otimes H^0 (C, E^\vee)^\vee \ \to \ 0
\]
where $E_1$ is ample (and in particular $h^0 (C, E_1^\vee) = 0$). Moreover, there is a direct sum decomposition $V = V_1 \oplus V_0$ where $V_1 = V \cap H^0 (C, E_1)$ and $V_0 \cong H^0 (C, E^\vee)^\vee$.
\end{enumerate}
\end{lemma}

\begin{proof}
Using the evaluation maps $\ev_{E, V}$ and $\ev_{E^\vee}$, we obtain a composed map
\[
\Oc \otimes H^0 (C, E^\vee ) \ \xrightarrow{\ev_{E^\vee}} \ E^\vee \ \xrightarrow{\left( \ev_{E, V} \right)^\vee} \ \Oc \otimes V^\vee ,
\]
which we denote by $\eta$. Clearly $\Image (\eta)$ and $\Ker (\eta)$ are trivial bundles. Moreover, as $(E, V)$ is generated, $\ev_{E, V}^\vee$ is injective; whence
\[
\Ker (\eta) \ = \ \Ker \left( \ev_{E^\vee} \right) .
\]
But this is exactly the kernel bundle $M_{E^\vee}$ of $(E^\vee , H^0 (E^\vee))$; cf.\ (\ref{DSBseq}). In particular, $h^0 ( C, \Ker(\eta) ) = 0$. As we have seen that $\Ker (\eta)$ is a trivial bundle, we have $\Ker (\eta) = 0$, whence $\ev_{E^\vee}$ is injective. Statement (a) follows.

For the rest: By part (a), the composition
\[
\tau \ := \ {^t\eta} \ = \ \left( \ev_{E^\vee} \right)^\vee \circ \ev_{E, V} \colon \Oc \otimes V \ \to \ \Oc \otimes H^0 (C, E^\vee )^\vee
\]
is a surjective map of trivial vector bundles. Let $V_1 \subseteq V$ be the subspace such that $\Ker (\tau) = \Oc \otimes V_1$. Now there is an exact diagram
\begin{equation} \label{evDiagram}
\xymatrix{
\Oc \otimes V_1 \ar[r] \ar[d] & \Oc \otimes V \ar[r] \ar[d]_{\ev_{E, V}} \ar[dr]^\tau & \Oc \otimes H^0 (C, E^\vee )^\vee \ar[d]^\Iden \\
 E_1 \ar[r] & E \ar[r]^-{\left( \ev_{E^\vee} \right)^\vee} & \Oc \otimes H^0 (C, E^\vee )^\vee
}
\end{equation}
where $E_1 := \Ker \left( \left( \ev_{E^\vee} \right)^\vee \right)$. From (a) it follows that $\left(\ev_{E^\vee} \right)^\vee$ is a surjective vector bundle map, and so $E_1$ is a vector subbundle of $E$ (that is, saturated). 
 Let $V_0$ be any subspace of $V$ of dimension $h^0 (C, E^\vee )$ which is complementary to $V_1$. Then
\[
\tau|_{\Oc \otimes V_0} \colon \Oc \otimes V_0 \ \to \ \Oc \otimes H^0 (C, E^\vee )^\vee
\]
is an isomorphism, and $\ev_{E, V} \circ \left( \tau|_{\Oc \otimes V_0} \right)^{-1}$ defines a splitting of the lower row of (\ref{evDiagram}).  
 Dualising the lower row and taking global sections, we obtain the sequence
\[ 0 \ \to \ H^0 (C, E^\vee ) \ \xrightarrow{\sim} \ H^0 (C, E^\vee ) \to H^0 (C, E_1^\vee ) \ \to \ H^0 (C, E^\vee ) \otimes H^1 (C, \Oc ) \ \to \ \cdots \]
But since we have just seen that the extension is trivial, the coboundary map is zero and $h^0 (C, E_1^\vee ) = 0$. 

Next: As $\ev_{E, V}$ is surjective, by the snake lemma the vertical map $\Oc \otimes V_1 \to E_1$ has zero cokernel. Thus $V_1$ generates $E_1$.

For ampleness: As $E_1$ is generated, every quotient of $E_1$ is generated and has nonnegative degree. As $h^0 (C, E_1^\vee ) = 0$, in fact every quotient has positive degree. Thus by \cite[Theorem 2.4]{Har71}, the bundle $E_1$ is ample (here we use the fact that the base field is $\C$).

Lastly, if $s \in V \setminus V_1$ then $s$ has nonzero image in $V/V_1 = H^0 (C, E^\vee)^\vee$. Therefore, $s$ takes values outside $E_1 = \Ker \left( E \to \Oc \otimes H^0 (C, E^\vee )^\vee \right)$ at a general point of $C$. We deduce that $V_1 = V \cap H^0 (C, E_1 )$. This completes the proof of (b).
\end{proof}

The following notation introduced in \cite{CH} will be useful.

\begin{definition}
Let $(F, W)$ be a generated coherent system with $\deg F > 0$. We set
\[
\lambda (F, W) \ := \ \frac{\deg F}{\dim W - \rank F} .
\]
This number will be called the \textsl{linear slope} of $(F, W)$. We abbreviate $\lambda (F, H^0 (C, F))$ to $\lambda (F)$. Note that $\lambda (F, W)$ is precisely the slope of the dual span bundle $\mfwd$.
\end{definition}

\begin{corollary} \label{LinSlopeEquality}
Let $(F, W)$ be a generated coherent system of type $(r_F, d_F, n_F)$. Suppose that $d_F > 0$. Then for any $t \ge 0$, we have
\begin{equation} \label{lambdaAmple}
\lambda (F, W) \ = \ \lambda \left( F \oplus \cO_C^{\oplus t} , W \oplus H^0 (C, \Oc^{\oplus t}) \right) .
\end{equation}
In particular, if $(E, V)$ is generated of type $(r, d, n)$ with $d > 0$, then $\lambda (E, V) = \lambda (E_1, V_1)$, where $E_1$ is the ample summand of $E$ and $V_1 = V \cap H^0 (C, E_1)$ is as in Lemma \ref{StructureGeneratedCohSys}.
\end{corollary}

\begin{proof}
We compute:
\[
\lambda \left( F \oplus \Oc^{\oplus t} , W \oplus H^0 \left( C, \Oc^{\oplus t} \right) \right) \ = \ \frac{\deg F}{(\dim W + t) - (\rank F + t)} \ = \ \lambda (F, W) .
\]
For the rest: By Lemma \ref{StructureGeneratedCohSys} (b), we have
\[
(E, V) \ \cong \ \left( E_1 \oplus \left(\Oc \otimes H^0 (C, E^\vee )^\vee \right) , ( V \cap H^0 (C, E_1) ) \oplus H^0 (C, E^\vee)^\vee \right) 
\]
By (\ref{lambdaAmple}), we obtain $\lambda (E, V) \ = \ \lambda (E_1, V_1)$ as desired.
\end{proof}

\section{Linear stability of scrolls and coherent systems} \label{Characterisation}

Linear stability of a subvariety $X \subset \PP^{n-1}$ was defined in the introduction. Let us now suppose $X$ to be a scroll given as the image of $\pe \to \pv$ where $E \to C$ is a vector bundle of positive degree and $V \subseteq H^0 ( \pe, \opeo )$ a generating subspace. In this section, we will give criteria for linear stability of $X$ in terms of defect of certain secants, and in terms of sections of elementary transformations of $E$.

For a subspace $W \subset V$, intuitively, $p_W ( X )$ is expected to have small reduced degree if the centre of projection $\PP W^\perp$ is a low-dimensional subspace containing many points of $X$; compare with the discussion on \cite[p.\ 56]{Mum}. Thus a linearly stable variety is expected not to have secants with large defect. However, any scroll $\pe$ contains $\pi$-defective subschemes (Definition \ref{DefnPiNondef}) with arbitrarily large defect which do not reflect any aspect of the global geometry of $\pe$; they occur for any $E$. One of our observations below is that linear semistability can be read off the behaviour of secants with ``honest'' defect; more precisely, secants which are defective but not $\pi$-defective.

\subsection{Characterisations of linear stability for scrolls}

\begin{proposition} \label{LinSstCrit} Let $E \to C$ be a vector bundle of degree $d$ and rank $r$. Let $\pi \colon \pe \to C$ be the projection. Let $V \subseteq H^0 (C, E)$ be a generating subspace of dimension $n$, and $\psi \colon \pe \to \pv$ the associated map. Then the following are equivalent.
\begin{enumerate}
\renewcommand{\labelenumi}{(\arabic{enumi})}
\item The variety $\psi (\pe)$ is Mumford linearly semistable in $\pv$.
\item For all elementary transformations $F \subset E$ with $\deg F > 0$ such that $V \cap H^0 (C, F)$ generates $F$, we have
\begin{equation} \label{TwoIneq}
\lambda \left( F, V \cap H^0 (C, F) \right) \ \ge \ \lambda (E, V) .
\end{equation}
\item For all $\pi$-nondefective length zero subschemes $Z \subset \pe$ such that $\Span \, \psi(Z)$ has dimension at most $n - r - 2$ and does not intersect a general fibre $\psi \left( \pe|_p \right)$, we have
\begin{equation} \label{ThreeIneq}
\frac{\length Z}{\dim \Span \psi (Z) + 1} \ \le \ \frac{d}{n - r} .
\end{equation}
\end{enumerate}
The equivalence also holds if we replace ``semistable'' with ``stable'' in $(1)$ and inequalities (\ref{TwoIneq}) and (\ref{ThreeIneq}) are made strict.
\end{proposition}

\begin{remark}
Notice that the number $\frac{d}{n-r}$ is exactly
\[
\frac{\length Z_0}{\dim \Span \psi (Z_0) + 1}
\]
for any $Z_0$ occurring as $\pe \times_{\pv} \PP W^\perp$ for a subspace $W$ of dimension $r$; that is, such that $\PP W^\perp$ computes $\deg \psi (\pe) \cdot \deg \psi$. Thus condition (3) can also be viewed as a slope inequality. (Note however that if $\dim W = r$ then $\dim p_W (\pe) < r$.)
\end{remark}

\begin{proof}
$(1) \Rightarrow (2)$: Let $F \subset E$ be an elementary transformation as stated. Write
\[
W \ := \ V \cap H^0 ( C, F ) .
\]
By Lemma \ref{DegreeProperties} (b) the map $p_W \circ \psi$ factorises via a morphism $\psi' \colon \PP F^\vee \to \PP W^\vee$. Now $\psi'$ is linear and hence injective on fibres, so the image has dimension at least $r - 1$. If it is not of dimension $r$, then by linearity of $\psi$ on fibres the image of $\psi'$ is a $\PP^{r-1}$. As $\psi ( \pe )$ is nondegenerate in $\pv$, also $\psi' ( \PP F^\vee )$ is nondegenerate in $\PP W^\vee$, whence $W$ has dimension $r$. But the only generated rank $r$ bundle with exactly $r$ sections is the trivial bundle $\Oc^{\oplus r}$. As we have assumed that $\deg F > 0$, this is excluded. We conclude that $p_W ( \psi (\pe) )$ has dimension $r$.

Next, write $\gamma := \deg \psi$ and $\delta := \deg p_W|_{\psi(\pe)}$. 
 By Lemma \ref{DegreeProperties} (b) and (c) we have
\[
\gamma \cdot \deg \psi ( \pe ) \ = \ d \quad \hbox{and} \quad \gamma \cdot \delta \cdot \deg \, p_W ( \pf ) \ = \ \deg F .
\]
As by hypothesis $(\pe, \pv)$ is linearly semistable, we have as desired
\[
\frac{\deg F}{\dim W - r} \ = \ \gamma \cdot \delta \cdot \reddeg p_W (\pf) \ \ge \ \gamma \cdot \reddeg (\pe) \ = \ \frac{d}{n - r} .
\]

$(2) \Rightarrow (3)$: Let $Z \subset \pe$ be a $\pi$-nondefective subscheme of length $e$ whose linear span is of dimension at most $n - r - 2$ and does not intersect a generic fibre of $\psi ( \pe )$. As in {\S} \ref{Zet}, write $E^Z := \pi_* \Iz ( 1 )$, a full rank subsheaf of $E$. Set
\[
W \ := \ V \cap H^0 (\pe , \Iz (1)) \ = \ V \cap H^0 \left( C, E^Z \right) .
\]
Now we have a diagram of sheaves
\begin{equation} \label{TwoImpliesThreeDiag} \xymatrix{
 & E^\vee \ar[d] \ar[r] & \left( E^Z \right)^\vee \ar[d] \\
 \Oc \otimes W^\perp \ar[r] & \Oc \otimes V^\vee \ar[r] & \Oc \otimes W^\vee .
} \end{equation}
By (\ref{DefnPsi}) and the discussion before it, we obtain correspondingly a diagram of varieties
\[ \xymatrix{
& \pe \ar@{-->}[r] \ar[d]^\psi & \PP \left( E^Z \right)^\vee \ar@{-->}[d] \\
\PP W^\perp \ar[r] & \pv \ar@{-->}[r] & \pw .
}\]
By Proposition \ref{SpanZ} (a), we have $\Span \, \psi ( Z ) = \PP W^\perp$. Therefore, since by hypothesis $\dim W^\perp \le n - r - 1$, we obtain $\dim W \ge r + 1$. Furthermore, as we have assumed that $\PP W^\perp$ does not meet a general fibre of $\pe$, the map $E^\vee|_x \to W^\vee$ is injective for a general $x \in C$. Thus $W$ is a generically generating subspace of $H^0 ( C, E^Z )$. Denote by $E^Z_W$ the subsheaf of $E^Z$ generated by $W$. As $h^0 (C, E^Z_W) \ge \dim W \ge r + 1$, we have $\deg E^Z_W > 0$.
 We now write
\[
m \ := \ \dim W \ = \ \dim \left( V \cap H^0 \left( C, E^Z \right) \right) \ = \ \dim \left( V \cap H^0 \left( C, E^Z_W \right) \right) .
\]
By $\pi$-nondefectivity $\deg E^Z = d - e$, and clearly $\deg E^Z_W \le d - e$. Applying hypothesis (2) to the elementary transformation $E^Z_W \subset E$, we obtain
\[
\frac{d-e}{m - r} \ \ge \ \frac{\deg E^Z_W}{m - r} \ \ge \ \frac{d}{n - r} .
\]
One computes that this is equivalent to $\frac{d}{n - r} \ge \frac{e}{n - m}$. 
 As $\Span \psi ( Z ) = \PP W^\perp$, in particular $\dim \Span \psi ( Z ) = n - m - 1$, and the above inequality becomes $\frac{d}{n-r} \ge \frac{\length Z}{\dim \Span \psi (Z) + 1}$ as desired.

$(3) \Rightarrow (1)$: Let $W \subseteq V$ be a subspace of dimension $m$ such that $p_W ( \pe )$ has dimension $r$. As $\psi$ is linear on fibres, the latter implies that
\begin{equation} \label{IntersEmpty}
\hbox{$\PP W^\perp \cap \psi ( \pe|_p )$ is empty for generic $p \in C$.}
\end{equation}
It follows that $W$ generically generates $E$. Let $E_W \subset E$ be the elementary transformation generated by $W$. Suppose $\deg E_W = d - e$. By Theorem \ref{F=EZ}, we may choose a $\pi$-nondefective subscheme $Z \subset \pe$ of length $e$ such that $E_W = E^Z$; and moreover, the space $\Span \psi (Z)$ is independent of the choice of $Z$ with this property and coincides with $\PP W^\perp$.

Now since $p_W ( \pe )$ has dimension $r$, we have $m \ge r + 1$. Therefore, $W^\perp$ has dimension $n - m \le n - r - 1$. By the last paragraph,
\begin{equation} \label{DimSpanSmallEnough}
\dim \Span \psi (Z) \ = \ n - m - 1 \ \le \ n - r - 2 .
\end{equation}
By (\ref{IntersEmpty}), (\ref{DimSpanSmallEnough}) and hypothesis (3), we obtain $\frac{d}{n - r} \ge \frac{e}{n - m}$. As stated above, this is equivalent to $\frac{d - e}{m - r} \ge \frac{d}{n - r}$. Now
\[
d \ = \ \deg E \ = \ ( \deg \psi ) \cdot \deg \left( \psi ( \pe ) \right)
\]
by Lemma \ref{DegreeProperties} (c). Moreover, by Lemma \ref{DegreeProperties} (b) we have a diagram
\[ \xymatrix{
\pe \ar@{-->}[r] \ar[d]_\psi & \PP E_W^\vee \ar[d]^{\tilde{\psi}} \\
\pv \ar@{-->}[r]_{p_W} & \pw
} \]
where $\tilde{\psi}$ is a morphism. By commutativity of this diagram and by Lemma \ref{DegreeProperties} (c), we obtain
\[
d - e \ = \ \deg E_W \ = \ (\deg \tilde{\psi} ) \cdot \deg \tilde{\psi} \left( \PP E_W^\vee \right) \ = \ \deg \left( p_W|_{\pe} \right) \cdot ( \deg \psi ) \cdot \left( \deg \, \psi ( \pe ) \right) .
\]
Putting these facts together, we obtain
\[
\frac{(\deg \psi) \cdot \deg \left( p_W|_{\pe} \right) \cdot \deg \left( \psi ( \pe ) \right)}{m - r} \ \ge \ \frac{(\deg \psi) \cdot \deg \left( \psi ( \pe ) \right)}{n - r} ;
\]
which yields $\reddeg p_W ( \psi ( \pe )) \ge \reddeg \psi ( \pe )$. We conclude that $\psi (\pe)$ is linearly semistable.

The last statement follows by changing ``$\ge$'' to ``$>$'' in the appropriate places.
\end{proof}

\noindent The following slight refinement of the above proposition will be useful in {\S} \ref{ExistenceLinSt}.

\begin{lemma} \label{CurvilSuffices}
Condition (3) of Proposition \ref{LinSstCrit} holds if and only if (\ref{ThreeIneq}) is satisfied for all \emph{curvilinear} $Z$ with the stated properties.
\end{lemma}

\begin{proof}
Let $Z \subset \pe$ be a arbitrary $\pi$-nondefective subscheme. By Theorem \ref{F=EZ}, there exists a $\pi$-nondefective curvilinear $Z'$ such that $E^Z = E^{Z'}$. By $\pi$-nondefectivity,
\[
\length Z' \ = \ \deg E - \deg E^{Z'} \ = \  \deg E - \deg E^Z \ = \ = \length Z .
\]
Moreover, $\Span \psi (Z') = \Span \psi (Z)$ by Proposition \ref{SpanZ} (b). Therefore,
\[
\frac{\length Z}{\dim \Span \psi (Z) + 1} \ = \ \frac{\length Z'}{\dim \Span \psi (Z') + 1} .
\]
Thus $Z$ satisfies (\ref{ThreeIneq}) if and only if $Z'$ does. The lemma follows.
\end{proof}

We mention a useful special case, which generalises \cite[Proposition 8.1]{MS} to the case of a hypersurface scroll, and gives a partial converse to \cite[Proposition 2.4]{BT}.

\begin{corollary} \label{LinStHypersurface} Suppose $n = r + 2$, so that $\pe$ is a hypersurface in $\PP^{n-1}$. Then $( E, V )$ is linearly semistable if and only if $\length Z \le \frac{\deg E}{2}$ for any $\pi$-nondefective $Z \subset \pe$ which is contracted by $\psi$. In particular, this applies if $\mult_{\psi(p)} \left( \psi ( \pe ) \right) \le \frac{\deg E}{2}$ for all $p \in \pe$. 
 Moreover, $( E, V )$ is linearly stable if and only if inequality is strict for all such $Z$. \end{corollary}

\begin{proof} Here $n - r = 2$. In view of condition (3) of Proposition \ref{LinSstCrit}, linear semistability of $(E, V)$ is equivalent to the condition that $\frac{d}{2} \le \length Z$ for all $\pi$-nondefective $Z$ such that $\Span \psi (Z)$ is of dimension zero; that is, $\psi ( Z )$ is a point. In particular, this follows if $\mult_{\psi(p)} \left( \psi ( \pe ) \right) \le \frac{\deg E}{2}$ for all $p \in \pe$. The proof of the condition for linear stability is identical. \end{proof}

\subsection{Coherent systems}

The first application of Proposition \ref{LinSstCrit} will be to characterise linear stability of a scroll $\pe \to \pv$ in terms of the coherent system $(E, V)$. Let us give some context. Let $L \to C$ be an ample line bundle and $V \subseteq H^0 (C, L)$ a generating subspace. Linear stability of the image of $C \to \pv$ has an equivalent definition in terms of the cohomology of subseries of $\PP V$. To our knowledge, this first appeared in \cite[Definition 2.1]{BS}, where it was used to study surfaces fibred over curves. This formulation has appeared subsequently in \cite{MS}, \cite{CT}, \cite{CMT} and elsewhere, and may be generalised to higher rank as follows.

In dealing with linear semistability we will require a more general notion of \textsl{subsystem} than is usual:

\begin{definition} \label{CohSubsys}
Let $(E, V)$ be a coherent system. A \textsl{coherent subsystem} of $(E, V)$ is a pair $(F, W)$ where $F \subseteq E$ is a subsheaf and $W \subseteq H^0 (C, F) \cap V$ is a subspace. Note that $F$ may be nonsaturated, and we may have $\rank F = \rank E$. A coherent subsystem $(F, W)$ is \textsl{generated} if it is generated as a coherent system; that is, the evaluation map $\Oc \otimes W \to F$ is surjective. 
\end{definition}

We can now recall the notion of linear semistability of a generated coherent system from \cite[Definition 1.2]{CH}.

\begin{definition} \label{DefnLinSstCohSys}
Let $(E, V)$ be a generated coherent system of type $(r, d, n)$ on $C$, where $r \ge 1$ and $d > 0$. Then $(E, V)$ is said to be \textsl{linearly semistable} if for all generated coherent subsystems $(F, W)$ where $\deg F > 0$ and $E/F$ is not a trivial sheaf, we have $\lambda (F, W) \ge \lambda (E, V)$; that is,
\begin{equation} \label{LinSstIneq}
\frac{\deg F}{\dim W - \rank F} \ \ge \ \frac{d}{n-r} .
\end{equation}
If inequality is strict for all such $(F, W)$, then $(E, V)$ is said to be \textsl{linearly stable}.
\end{definition}

\begin{remark}
For $\alpha$-semistability (Definition \ref{alphaSt}), the inequality (\ref{alphaSlopeIneq}) is only required for coherent subsystems $(F, W)$ where $F$ is a proper subbundle of $E$. However, linear stability depends strongly on the behaviour of nonsaturated subsheaves; in particular, elementary transformations of $E$ (cf.\ (2) in Proposition \ref{LinSstCrit}). This is why we must use a more general notion of subsystem. See also \cite[Remark 2.7]{BO}.
\end{remark}

We will now prove that linear (semi)stability of $(E, V)$ in the above sense is equivalent to Mumford linear (semi)stability of the image of $\pe \to \pv$. This is a slight modification of the claim made in \cite[Remark 5.18]{CHL}, and will justify the use of the term ``linearly (semi)stable'' for generated coherent systems.

\begin{theorem} \label{EquivDefns}
Let $(E, V)$ be a generated coherent system of type $(r, d, n)$ and $\psi \colon \pe \to \pv$ the associated map to projective space. Then the scroll $\psi ( \pe )$ is linearly (semi)stable in the sense of Definition \ref{MumfordLinSst} if and only if the coherent system $(E, V)$ is linearly (semi)stable in the sense of Definition \ref{DefnLinSstCohSys}.
\end{theorem}

\begin{proof}
Suppose that $(E, V)$ is linearly semistable in the sense of Definition \ref{DefnLinSstCohSys}. Then in particular, $\lambda (F, W) \ge \lambda(E, V)$ for all generated subsystems $(F, W)$ where $\rank F = \rank E$ and $\deg F > 0$. By implication (2) $\Rightarrow$ (1) of Proposition \ref{LinSstCrit}, we see that $\psi (\pe)$ is Mumford linearly semistable in $\pv$.

Conversely, suppose that $\psi (\pe)$ is Mumford linearly semistable in $\pv$. Let $(F, W)$ be a proper generated coherent subsystem of $(E, V)$ such that $\deg F > 0$ and $E/F$ is not trivial. If $\rank F = r$, then by implication (1) $\Rightarrow$ (2) of Proposition \ref{LinSstCrit} we have $\lambda (F, W) \ge \frac{d}{n-r}$.

Suppose that $\rank F < r$. As $V$ generates $E$, we may choose a subspace $W_1 \subset V$ of dimension $r - \rank F$ intersecting $W$ in zero such that $W \oplus W_1$ generically generates $E$. 
 Then $F \oplus (\Oc \otimes W_1 )$ is an elementary transformation of $E$ which is generated by the subspace $W \oplus W_1$ of $V$. As $E/F$ is not trivial, $F \oplus (\Oc \otimes W_1) \ne E$. Therefore, by implication (1) $\Rightarrow$ (2) of Proposition \ref{LinSstCrit} and using Corollary \ref{LinSlopeEquality}, we obtain
\[
\lambda \left( F, V \cap H^0 (C, F) \right) \ = \ \lambda \left( F \oplus ( \Oc \otimes W_1 ) , \left( H^0 (C, F) \cap V \right)  \oplus W_1 \right) \ \ge \ \lambda (E, V) .
\]
As $\dim W \le \dim \left( H^0 (C, F) \cap V \right)$, this implies that $\lambda (F, W) \ge \lambda (E, V)$. We conclude that $(E, V)$ is a linearly semistable coherent system.

The equivalence of the respective notions of linear \emph{stability} can be shown in exactly the same way.
\end{proof}

\begin{remark} \label{TrivSummand}
In the above proof, suppose that $W_1$ is nonzero. Then we have a diagram
\[ \xymatrix{
0 \ar[r] & \mfw \ar[d] \ar[r] & \Oc \otimes W \ar[r] \ar[d] & F \ar[r] \ar[d] & 0 \\
 0 \ar[r] & M_{F \oplus (\Oc \otimes W_1), W \oplus W_1} \ar[r] & \Oc \otimes ( W \oplus W_1 ) \ar[r] \ar[d] & F \oplus (\Oc \otimes W_1) \ar[r] \ar[d] & 0 \\
 & & \Oc \otimes W_1 \ar[r] & \Oc \otimes W_1 & } \]
By the snake lemma, $M_{F \oplus (\Oc \otimes W_1), W \oplus W_1} \cong \mfw$. Thus adding a trivial summand to $(F, W)$ does not change $\mfw$. Compare with Corollary \ref{LinSlopeEquality}.
\end{remark}

\begin{remark} \label{ShrinkingCentre}
Geometrically, extending $W$ to $W \oplus W_1$ corresponds to shrinking the centre of projection in such a way that $p_{W \oplus W_1} ( \pe )$ has dimension $r$ while the reduced degree of the image is preserved. More precisely: Consider the diagram
\[ \xymatrix{
& \pe \ar@{-->}[rr] \ar@{-->}[dr] \ar[dd]_\psi & & \pf \ar[dd]^\rho \\
\PP \left( W \oplus W_1 \right)^\perp \ar[dr] \ar[d] & & \PP \left( F \oplus (\Oc \otimes W_1) \right)^\vee \ar@{-->}[ur] \ar[dd]_(.3){\tilde{\rho}} & \\
\PP W^\perp \ar[r] & \pv \ar@{-->}[rr]^(.7){p_W} \ar@{-->}[dr]_{p_{W \oplus W_1}} & & \pw \\
&  & \PP \left( W \oplus W_1 \right)^\vee \ar@{-->}[ur] &
} \]
By Lemma \ref{DegreeProperties} (b), the maps $\psi$, $\rho$ and $\tilde{\rho}$ are morphisms, and in particular $\tilde{\rho}$ has image of dimension $r$. We suppose for simplicity that they are of degree $1$. Then using Lemma \ref{DegreeProperties} (c), we obtain
\[
\frac{\deg \tilde{\rho} \left( \PP \left( F \oplus ( \Oc \otimes W_1 ) \right)^\vee \right)}{m + (r - \rank F) - r} \ = \ \frac{\deg \rho \left( \pf \right)}{m - \rank F} ;
\]
that is, $\reddeg \tilde{\rho} \left( \PP \left( F \oplus ( \Oc \otimes W_1 ) \right)^\vee \right) = \reddeg \rho ( \pf )$.

The above situation illustrates why we allow coherent systems in which the ambient bundle may not be ample. When the direct sum decomposition $W \oplus W_1$ is chosen, we obtain maps $\pf \to \pw \hookrightarrow \PP (W \oplus W_1)^\vee$, and the subvariety $\tilde{\rho} ( \PP \left( F \oplus \left( \Oc \otimes W_1 ) \right)^\vee \right)$ is the cone over $\rho ( \pf )$ with vertex $\PP W_1^\vee$.
\end{remark}

\section{Existence of linearly stable projective models} \label{ExistenceLinSt}

Let $E$ be any vector bundle of rank $r \ge 2$. In this section we construct a linearly stable model of $\pe$ in $\PP^{n-1}$ for any $n \ge r + 2$. We firstly prove some technical lemmas.

Recall that for each vector bundle $G$, we have a unique Harder--Narasimhan filtration
\begin{equation} \label{HNfilt}
0 \ = \ G^0 \ \subset \ G^1 \ \subset \ G^2 \ \subset \ \cdots \ \subset G^s \ = \ G
\end{equation}
characterised by the property that $G^i / G^{i-1}$ is semistable for $1 \le i \le s$, and if $s \ge 2$ then $\mu \left( G^i / G^{i-1} \right) > \mu \left( G^{i+1} / G^i \right)$ for $1 \le i \le s-1$. Following \cite{But94}, we write
\begin{equation} \label{muplus}
\mu^+ (G) \ := \ \mu (G^1) \ = \ \max\{ \mu(H) : H \hbox{ a nonzero subbundle of } G \} .
\end{equation}
If $G$ is semistable, then $\mu^+ (G) = \mu (G)$. Moreover, we let
\[
\HN (G) \ := \ \bigoplus_{i=1}^s \frac{G^i}{G^{i-1}}
\]
be the associated graded bundle of the Harder--Narasimhan filtration.

\begin{lemma} \label{HNsections}
Let $G \to C$ be any vector bundle. Then $h^0 (C, G) \le h^0 (C, \HN (G))$.
\end{lemma}

\begin{proof}
If $s = 1$ then $G = \HN(G)$ and there is nothing to prove. If $s \ge 2$, consider the filtration
\begin{equation} \label{HNexactseq}
0 \ = \ G^2 / G^1 \ \subset \ G^3 / G^1 \ \subset \ \cdots \ \subset G^s / G^1 \ = \ G / G^1
\end{equation}
It is easy to check that this is the Harder--Narasimhan filtration of $G / G^1$. 
 Therefore,
\[
\HN \left( \frac{G}{G^1} \right) \ = \ \bigoplus_{i=2}^s \frac{G^i}{G^{i-1}} .
\]
By induction on $s$, we may assume that $h^0 (C, G / G^1 ) \le h^0 (C, \HN ( G / G^1 ))$. Then from the exact sequence $0 \to G^1 \to G \to G/G^1 \to 0$, we obtain as desired
\begin{multline*}
h^0 (C, G) \ \le \ h^0 (C, G^1) + h^0 \left(C,  G / G^1 \right) \ \le \\
h^0 (C, G^1 ) ) + h^0 \left(C, \HN \left(G / G^1\right) \right) \ = \ h^0 (C, \HN ( G ) ) . \qedhere
\end{multline*}
\end{proof}

The next lemma is extremely elementary.

\begin{lemma} \label{calculus}
Let $r$ and $n$ be integers with $n \ge r + 2$. Then
\[
\tag{a}
\max \left\{ \frac{(n - r)(n - t)}{n - r - t} : 1 \le t \le n - r - 1 \right\} \ = \ (n - r)(r + 1)
\]
and
\[ \tag{b}
\max \left\{ \frac{t(n - t)}{n - r - t} : 1 \le t \le n - r - 1 \right\} \ = \ (n - r - 1)(r + 1) .
\]
\end{lemma}

\begin{proof}
We easily check that the function $\frac{(n - r)(n - x)}{n - r - x}$ is increasing for $0 < x < n-r$, and then (a) follows. As $\frac{(n - r)(n - x)}{n - r - x}$ and $\frac{x}{n - r}$ are positive and increasing for these values of $x$, using the Leibniz rule we see that their product is also, and then (b) follows.
\end{proof}

We prove now a statement which will be required in the proof of Lemma \ref{pinondefSpan}.

\begin{lemma} \label{SchThSupp}
Let $\phi \colon F \to T$ be a map of $\Oc$-modules where $T$ is torsion. Let $\cI$ be the ideal sheaf $\Ann (T)$. Then $\cI \cdot F \subseteq \Ker (\phi)$. Moreover, if $F$ is locally free, the multiplication map $\cI \otimes F \to \cI \cdot F$ is an isomorphism.
\end{lemma}

\begin{proof}
As $C$ is a curve, $T$ is supported at finitely many points. Thus it suffices to prove the statements for the restriction of $F$ to an open affine subset $\Spec A \subset C$ containing the support of $T$. Therefore, we may assume that $F = \tM$ and $T = \tN$ for $A$-modules $M$ and $N$. Write $I$ for the ideal $\Ann (N) \subset A$. Then if $i \in I$ and $f \in F$, we have $\phi (i \cdot f) = i \cdot \phi (f) = 0$. The first statement follows. For the rest: If $F$ is locally free then $M$ is flat, and the sequence $0 \to I \otimes M \xrightarrow{m} M \to \frac{A}{I} \otimes M \to 0$ is exact, and clearly $\Image (m) = I \cdot M$. 
\end{proof}

Next, we recall the definition of $(k-1)$-very ampleness.

\begin{definition} \label{kVeryAmple}
For $k \ge 1$, a bundle $F$ of rank $r$ is said to be \textsl{$(k-1)$-very ample} if for each effective divisor $D$ of degree $k$ we have $h^0 (C, F(-D)) \ = \ h^0 (C, F) - kr$.
\end{definition}

\begin{remark} \label{kVAkmoVA}
Notice that $E$ is $0$-very ample if and only if it is generated; and that a $(k-1)$-very ample sheaf is also $(k-i)$-very ample for $1 \le i \le k$.
\end{remark}

\noindent We will need the following technical fact about relatively nondefective subschemes (cf.\ Definition \ref{DefnPiNondef}).

\begin{lemma} \label{pinondefSpan}
Let $F$ be an $(e - 1)$-very ample bundle and $\pi \colon \PP F^\vee \to C$ the associated projective bundle. If $Z \subset \PP F^\vee$ is $\pi$-nondefective of length $e$, then $\dim \Span \left( \Psi (Z) \right) = e - 1$.
\end{lemma}

\begin{proof}
Since $Z$ is $\pi$-nondefective, the sequence $0 \to F^Z \to F \xrightarrow{\gamma} \pi_* \Oz (1) \to 0$ is exact. By Proposition \ref{SpanZ}, therefore, $\Span \Psi (Z)$ is the space
\[
\ \PP \Ker \left( H^0 (C, F)^\vee \ \to \ H^0 \left( C, F^Z \right)^\vee \right) \ = \ \PP \Image \left( H^0 (C, \pi_* \Oz (1))^\vee \ \xrightarrow{H^0 (\gamma)^\vee} \ H^0 (C, F)^\vee \right) .
\]
Thus it would suffice to show that $H^0 (\gamma) \colon H^0 (C, F) \to H^0 (C, \pi_* \Oz(1) )$ is surjective.

Let $D$ be the scheme-theoretic support of $\pi_* \Oz (1)$; that is, the subscheme of $C$ defined by the ideal sheaf $\Ann \left( \pi_* \Oz (1) \right)$. By Lemma \ref{SchThSupp}, we have $F(-D) \subseteq F^Z$. As $Z$ has length $e$, so does $\pi_* \Oz (1)$, and so $\deg D \le e$. We obtain a diagram
\[ \xymatrix{
F(-D) \ar[r] \ar[d] & F \ar[r]^-\alpha \ar[d]^\Iden & F/F(-D) \ar[d]^\beta \\
F^Z \ar[r] & F \ar[r]^-\gamma & \pi_* \Oz (1)
} \]
This shows that $H^0 (\gamma)$ factorises as
\[
H^0 (C, F) \ \xrightarrow{H^0 (\alpha)} \ H^0 (C, F/F(-D)) \ \xrightarrow{H^0 (\beta)} \ H^0 (C, \pi_* \Oz (1)) .
\]
Now since $Z$ is $\pi$-nondefective, $\gamma = \beta \circ \alpha$ is surjective, and in particular $\beta$ is surjective. As $\beta$ is a map of sheaves with support of dimension zero, also $H^0 (\beta)$ is surjective. Furthermore, $H^0 (\alpha)$ is surjective because $F$ is $(e - 1)$-very ample. Thus $H^0 (\gamma)$ is surjective, as desired.
\end{proof}

\begin{proposition} \label{one}
Suppose that $F \to C$ is $(n - r - 1)(r + 1)$-very ample. Let $V \subset H^0 (C, F)$ be a general subspace of dimension $n$, and let $\psi \colon \PP F^\vee \to \PP V^\vee$ be the natural map. For $1 \le t \le n - r - 1$, if $Z \subset \PP F^\vee$ is a curvilinear $\pi$-nondefective subscheme satisfying $\length (Z) > \frac{t(n - t)}{n - r - t}$, then $\dim \Span \psi (Z) \ge t$.
\end{proposition}

\begin{proof}
Fix $t \in \{ 1 , \ldots , n - r - 1 \}$. We begin by proving the statement for $Z$ of length
\[
e \ := \ \min \left\{ e' \in \bZ : e' > \frac{t(n - t)}{n - r - t} \right\} \ = \
\begin{cases}
\frac{t(n - t)}{n - r - t} + 1 \hbox{ if } (n-r-t) | t(n-t) ; \\
\left\lceil \frac{t(n - t)}{n - r - t} \right\rceil \hbox{ otherwise.}
\end{cases}
\]
Set $N := h^0 (C, F)$, and write
\[
\Psi \colon \PP F^\vee \ \to \ \PP H^0 (C, F)^\vee \ \cong \ | \opfo |^\vee \ = \ \PP^{N - 1}
\]
for the natural map. By the very ampleness assumption together with Lemma \ref{calculus} (b) and Remark \ref{kVAkmoVA}, we see that $F$ is $(e - 1)$-very ample. Therefore, $\dim \Span \Psi (Z) \ = \ e - 1$ by Lemma \ref{pinondefSpan}. It follows that if $\dim \Span \psi (Z) \le t - 1$, the cone $V^\perp$ over the centre of projection fits into a diagram of the form
\[ \xymatrix{
0 \ar[r] & \Pi_1 \ar[r] \ar[d]^\Iden & V^\perp \ar[r] \ar[d] & \Pi_2 \ar[r] \ar[d] & 0 \\
0 \ar[r] & \Pi_1 \ar[r] & H^0 (C, F)^\vee \ar[r] & \frac{H^0 (C, F)^\vee}{\Pi_1} \ar[r] & 0
} \]
where $\Pi_1$ is an $(e - t)$-dimensional subspace of the cone over $\Span \Psi (Z)$ and $\Pi_2$ has dimension $(N - n) - (e - t)$. The spaces $V^\perp$ obtained in this way define a locus in $\Gr (N - n, N)$ of dimension at most
\[
\dim \Gr (e - t, e) + \dim \Gr ((N - n) - (e - t), N - (e - t)) + \dim \Hilb^e (\PP F^\vee)_\cl ,
\]
where $\Hilb^e ( \PP F^\vee )_\cl$ denotes the locus of length $e$ curvilinear subschemes. As this is an open subset of the smoothable component of $\Hilb^e ( \PP F^\vee )$, it has dimension $re$. Thus the above number is
\[
(e - t)t + (N - n)n - (e - t)n + re \ = \ \dim \Gr (N - n, N) - \left( (e - t)(n - t) - re  \right) .
\]
As by hypothesis $e > \frac{t(n - t)}{n - r - t}$, we have $(n - t)(e - t) - re > 0$. Therefore, the locus of such $V^\perp$ is of positive codimension in $\Gr (N - n, N)$. Thus for a general $V \in \Gr (n, N)$ we have $\dim \Span \psi (Z) \ge t$ for every curvilinear $\pi$-nondefective $Z$ of length $e$.

Suppose now that $\length Z > e$. By curvilinearity, we may choose a length $e$ subscheme $Z_1 \subset Z$ which is again curvilinear. Moreover, on $\pe$ we have an exact sequence
\[
0 \ \to \ \frac{\cI_{Z_1} (1)}{\cI_Z (1)} \ \to \ \Oz (1) \ \xrightarrow{\rho} \ \cO_{Z_1} (1) \ \to \ 0 .
\]
As these sheaves have support of dimension zero, in particular $R^1_{\pi_*} \left( \frac{\cI_{Z_1} (1)}{\cI_Z (1)} \right)$ is zero and $\pi_* ( \rho )$ is surjective. Since $Z$ is $\pi$-nondefective, also the composed map
\[
E \ \isom \ \pi_* \opeo \ \to \ \pi_* \Oz (1) \ \xrightarrow{\pi_* ( \rho )} \ \pi_* \cO_{Z_1} (1)
\]
is surjective, and thus $Z_1$ is also $\pi$-nondefective by definition. Now we apply the above argument to $Z_1$ and obtain $\dim \Span \psi (Z_1) \ge t$. Thus the same is true for $\Span \psi (Z)$.
\end{proof}

\begin{proposition} \label{VAmpleLinSt}
Suppose that $F \to C$ is a bundle of rank $r$ which is $(n - r - 1)(r + 1)$-very ample. Suppose in addition that $\deg F \ge (n - r)(r + 1)$. Then if $V \subset H^0 (C, F)$ is a general subspace of dimension $n$, the scroll $\psi ( \PP F^\vee ) \subset \pv$ is linearly stable.
\end{proposition}

\begin{proof}
Let $Z \subset \PP F^\vee$ be a $\pi$-nondefective curvilinear subscheme of dimension zero such that $\psi ( \PP F^\vee|_x ) \cap \Span \psi ( Z' )$ is empty for general $x \in C$. Suppose that $\Span \psi (Z)$ has dimension $t - 1$, where $1 \le t \le n - r - 1$. 
 As $V$ is general, by Proposition \ref{one} we may assume that $\length (Z) < \frac{t(n - t)}{n - r - t}$. Furthermore, $\deg F \ge \frac{(n - t)(n - r)}{n - r - t}$ by Lemma \ref{calculus} (a) and our assumption on $\deg F$. Putting these inequalities together, we obtain
\[
\frac{\length (Z)}{t} \ < \ \frac{n - t}{n - r - t} \ \le \ \frac{\deg F}{n - r} .
\]
Therefore $(F, V)$ is linearly stable by Lemma \ref{CurvilSuffices} and implication (3) $\Rightarrow$ (1) of Proposition \ref{LinSstCrit}.
\end{proof}

\noindent Now we can prove the main result of this section.

\begin{theorem} \label{AnyELinSt}
Let $E$ be a vector bundle of rank $r$ and degree $d$, and let $\mu^+ ( E^\vee )$ be as defined in (\ref{muplus}). If
\begin{equation} \label{ConditionK}
k \ > \ \max \left\{ 2g - 2 + \mu^+ ( E^\vee ) + (n - r - 1)(r + 1) + 1, \frac{(n - r)(r + 1) - d}{r} \right\} ,
\end{equation}
then there exists a linearly stable generated coherent system of the form $(E \otimes L_1^k , V)$ where $L_1$ is a line bundle of degree $1$ and $V \subseteq H^0 (C, E \otimes L_1^k )$ is a subspace of dimension $n$. In particular, there exists a linearly stable projective model $\PP E^\vee \to \PP^{n-1}$ of degree $d + kr$ for any $n \ge r + 2$.
\end{theorem}

\begin{proof}
To ease notation, write $e_0 := (n - r - 1)(r + 1) + 1$. Let $L_1 \to C$ be any line bundle of degree $1$. Then $E \otimes L_1^k$ is $(e_0 - 1)$-very ample if
\begin{equation} \label{vanishing}
h^1 ( C, E \otimes L_1^k (-D) ) \ = \ h^0 ( C, \Kc \otimes E^\vee \otimes L_1^{-k} (D) ) \ = \ 0
\end{equation}
for each effective divisor $D$ of degree $e_0$. By Lemma \ref{HNsections}, this would follow if
\[
h^0 \left( C, \Kc \otimes \HN (E^\vee) \otimes L_1^{-k} (D) \right) \ = \ 0
\]
for all such $D$. Now by (\ref{ConditionK}) and the definition of the Harder--Narasimhan filtration (\ref{HNfilt}), each summand of $\Kc \otimes \HN (E^\vee) \otimes L_1^{-k} (D)$ is semistable of slope at most
\[
\mu \left( \Kc \otimes ( E^\vee )^1 \otimes L_1^{-k} (D) \right) \ = \ 2g - 2 + \mu^+ (E^\vee) - k + e_0 .
\]
By (\ref{ConditionK}), this is negative. The desired vanishing (\ref{vanishing}) follows, and we conclude that $E \otimes L_1^k$ is $(e_0 - 1)$-very ample.

Furthermore, (\ref{ConditionK}) also implies that $\deg E \otimes L_1^k = d + kr > (n - r)(r + 1)$. By Proposition \ref{VAmpleLinSt}, if $V$ is general in $\Gr (n, H^0 (C, E))$ then $(E \otimes L_1^k , V)$ is linearly stable.
\end{proof}

\begin{remark}
If $g \ge 2$, or $g \ge 1$ and $r \ge 2$, then a straightforward computation shows that (\ref{ConditionK}) is satisfied if $k > 2g - 2 + \mu^+ ( E^\vee ) + (n - r - 1)(r + 1) + 1$.
\end{remark}

\begin{remark}
Proposition \ref{VAmpleLinSt} may be viewed as a result of the same type as \cite[Theorem 1]{Rob}. In particular, the latter result gives a bound on multiplicity of a point of the image of a general projection of the $d$-uple embedding of a variety. The function of twisting with $L_1^{\otimes k}$ as above is analogous to that of composing with the $d$-uple embedding in \cite{Rob}; namely, to achieve an embedding which is sufficiently ample that one can bound the secant defect introduced by a general projection.
\end{remark}

\section{Linearly stable systems with nonsemistable kernel bundles} \label{Counterexamples}

As discussed in the introduction, there are by now several examples showing that linear stability of a generated $(E, V)$ is in general strictly weaker than slope stability of $\mev$. In this section, as part of an effort to clarify the exact relation between these types of stability, we give two classes of examples illustrating this fact.

\subsection{Systems with nonsemistable ambient bundle}

In \cite[{\S} 5.4]{CHL}, generalising a construction in \cite[{\S} 8]{MS}, a coherent system of type $(2, d, 4)$ is constructed over any curve showing that a linearly stable system may not have semistable dual span bundle. We use the above results to generalise this counterexample to systems of type $(r, d, n)$ for any $r \ge 1$ and $n \ge r + 2$.

\begin{example} \label{PullbackCounterex}
Let $E \to \PP^1$ be the bundle $\opo^{\oplus (r - 1)} \oplus \opo (1)$. Setting $g = 0$ and $\mu (E) = \frac{1}{r}$ and $\mu^+ ( E^\vee ) = 0$, following (\ref{ConditionK}) we choose $k$ satisfying
\begin{equation} \label{P1bound}
k \ > \ \max \left\{ (n - r - 1)(r + 1) - 2, \frac{(n - r)(r + 1) - 1}{r} \right\} .
\end{equation}
We suppose in addition that
\begin{equation} \label{kP1}
kr \ \not\equiv \ -1 \ \mod (n-r) .
\end{equation}
Then by Proposition \ref{VAmpleLinSt}, if $V$ is a general element of $\Gr (n, H^0 (\PP^1 E \otimes \opo (k) )$, the system $(E \otimes \opo (k) , V)$ is linearly stable of type $(r, 1 + kr, n)$. However, since every bundle on $\PP^1$ splits as a direct sum of line bundles, $\mev$ is semistable only if $\rank \mev = (n-r)$ divides $\deg \mev = 1 + kr$; this is excluded by hypothesis (\ref{kP1}).

Now let $C$ be any curve, and let $f \colon C \to \PP^1$ be a covering. Then $(f^* E, f^* V)$ is linearly stable by \cite[Lemma 5.12]{CHL}, but $M_{(f^* E, f^*V)} = f^* \mev$ is nonsemistable since $\mev$ is\footnote{The proof of \cite[Lemma 5.12]{CHL} is slightly incomplete for our purposes due to the inequivalence of Definition \ref{DefnLinSstCohSys} and \cite[Definition 5.1]{CHL}. This can be amended by noting that, with notation as in ibid., since $f$ is finite and in particular flat, $(f^* E)/(f^* E_W) \cong f^* \left( E / E_W \right)$, and if $f^* (E/E_W)$ is not trivial then neither is $E/E_W$, so the required linear slope inequality can be assumed to hold.}.
\end{example}

\begin{remark}
If $r$ and $n$ are even then (\ref{kP1}) is satisfied for all $k$. Also, if $n = 4$ and $r = 2$ the bound (\ref{P1bound}) is easily checked to be $k \ge 3$, and we recover the construction of \cite[Theorem 5.17]{CHL}.
\end{remark}

\subsection{Systems with stable ambient bundle}

In the above example, the bundle $E$ is itself nonsemistable. We will now show that any stable bundle $E$ of large enough degree appears in a generated coherent system $(E, V)$ which is linearly stable, but with $\mev$ not semistable. Firstly, we recall a well known result on ample line bundles which we will require.

\begin{lemma} \label{Ampleness}
Let $X$ be a projective variety and $\cL \to X$ a line bundle. Suppose $V$ is a generating subspace of $H^0 ( X, \cL )$. If $\cL$ is ample, then the associated map $\phi \colon X \to \PP V^\vee$ has finite fibres.
\end{lemma}

\begin{proof}
For complete linear systems, this is \cite[Corollary 1.2.15]{Laz1}. The general case is similar: Since $V$ is generating, $\phi$ is defined everywhere. Therefore, if $\phi$ contracts a closed curve $Y \subseteq X$, then the restriction of $\phi^* \cO_{\pv} ( 1 ) \cong \cL$ to $Y$ is trivial. In particular, $\cL|_Y$ is not ample. By \cite[Proposition 1.2.13]{Laz1}, $\cL$ is not ample on $X$.
\end{proof}

Next, we recall the \emph{gonality sequence} of $C$; see for example \cite[{\S} 4]{LN}. For $k \ge 1$, write $\delta_k (C)$ for the smallest degree of a line bundle $L \to C$ satisfying $h^0 (C, L) = k + 1$. 
 Then by \cite[Theorem VII.2.3]{ACGH}, we have
\begin{equation} \label{GonSeq}
\delta_k (C) \ \le \ \left\lceil \frac{kg}{k+1} + k \right\rceil .
\end{equation}
In fact we have equality if $C$ is general in moduli, but we will not require this here. Now we can give the main result of this subsection.

\begin{theorem} \label{AnyENotSst}
Let $C$ be any curve of genus $g \ge 2$. For any $r \ge 2$, there exists an integer $d_0 (r)$ such that if $E$ is any stable bundle of rank $r$ and degree $d \ge d_0 (r)$, then there exists a generating subspace $V \subseteq H^0 ( C, E )$ of dimension $r + 2$ such that $(E, V)$ is generated and linearly stable, but the rank two bundle $\mev$ is not semistable.
\end{theorem}

\begin{proof}
Let $S^\vee$ be a line bundle of degree $\delta_{r+1} (C)$ with $h^0 ( C, S^\vee ) = r + 2$. As $h^0 (C, L) < r+2$ for $\deg L < \delta_{r+1} (C)$, we see that $S^\vee$ is generated. Set
\[
F^\vee \ := \ M_{S^\vee} \ = \ \Ker \left( \Oc \otimes H^0 ( C, S^\vee ) \to S^\vee \right) .
\]
Then $F$ has rank $r + 1$ and degree $s$, and is generated by $V := H^0 ( C, S^\vee )^\vee$. Note that $F$ need not be stable.

Now by \cite[Remark, p.\ 18]{PR}, there exists an integer $d_1 (r)$ such that for $d \ge d_1 (r)$, \emph{any} stable bundle of rank $r$ and degree $d$ is a quotient of $F$. Set
\[
d_0 (r) \ := \ \max\{ 2 \cdot \delta_{r+1} (C) + 1 , d_1 (r) \} ,
\]
and assume that $d \ge d_0 (r)$. Let $E$ be any stable bundle of rank $r$ and degree $d$. Then there is a short exact sequence $0 \to N \to F \to E \to 0$ where $N$ is a line of degree $\delta_{r+1} (C) - d$. As $d \ge 2 \cdot \delta_{r+1} (C) + 1$, in particular $\deg N < 0$; whence $h^0 ( C, N ) = 0$. Therefore, we have a diagram
\begin{equation} \label{ButlerDiag} \xymatrix{
 & & N \ar[d] \\
 S \ar[r] \ar[d] & \Oc \otimes V \ar[d]^\Iden \ar[r] & F \ar[d] \\
 \mev \ar[r] \ar[d] & \Oc \otimes V \ar[r] & E \\
 N & &
} \end{equation}
where $\mev$ is the kernel bundle of $(E, V)$, a bundle of rank two and degree $-d$. Then
\[
\mu ( S ) \ = \ -\delta_{r+1} (C) \ > \ - \left( \frac{2 \cdot \delta_{r+1}(C) + 1}{2} \right) \ \ge \ -\frac{d}{2} \ = \ \mu \left( \mev \right) ,
\]
so $\mev$ is not semistable.

Now $(E, V)$ is a generated coherent system of type $(r, d, r+2)$. We claim that $(E, V)$ is linearly stable. As the image of $\psi \colon \PP E^\vee \to \PP V^\vee$ is a hypersurface in $\PP V^\vee$, by Corollary \ref{LinStHypersurface} it suffices to show that $\psi ( \PP E^\vee )$ has no points of multiplicity $\ge \frac{d}{2}$.

Now since $F$ is generated, every quotient of $F$ has nonnegative degree; and since $F^\vee = M_{S^\vee}$ is a kernel bundle, $h^0 (C, F^\vee) = 0$. Thus every quotient of $F$ has positive degree; whence $F$ is ample by \cite[Theorem 2.4]{Har71}. As $V$ is a generating subspace of $H^0 ( C, F ) = H^0 ( \PP F^\vee , \cO_{\PP F^\vee} ( 1 ) )$, the natural map
\[
\psi' \colon \PP F^\vee \ \to \ \PP V^\vee
\]
is defined everywhere. Thus $\psi'$ is quasi-finite by Lemma \ref{Ampleness}, and finite since the varieties are projective. By for example \cite[Exercise 9.3 (a)]{Har77}, moreover, $\psi'$ is flat, whence we see that all fibres are zero-dimensional subschemes of $\PP F^\vee$ of the same length. From Lemma \ref{DegreeProperties} (c) it follows that this length is $\deg F = \delta_{r+1} (C)$.

Now $\psi \colon \pe \to \pv$ is the restriction of $\psi'$ to the projective subbundle $\pe \subset \PP F^\vee$. Thus for any $\nu \in \pe$, the fibre $\psi^{-1} ( \psi (\nu) )$ is contained $(\psi')^{-1} (\psi (\nu) )$. As we have seen that this has dimension zero and length $\deg F = \delta_{r+1} (C)$, we deduce that
\[
\mult_{\psi (\nu)} \psi \left( \PP E^\vee \right) \ \le \ \delta_{r+1}(C) \ < \ \frac{2 \cdot \delta_{r+1}(C) + 1}{2} \ \le \ \frac{d}{2}
\]
for all $\nu \in \PP E^\vee$, as desired.

In summary, given any stable $E$ of degree $d \ge d_0 (r)$, we have exhibited a subspace $V \subseteq H^0 ( C, E )$ such that $(E, V)$ is generated and linearly stable, but $\mev$ is not semistable.
\end{proof}

\begin{remark} \quad \begin{enumerate} \renewcommand{\labelenumi}{(\alph{enumi})} \item Note that the subspace $V$ belongs to the image of $H^0 ( C, F ) \to H^0 ( C, E )$, which has codimension at least
\[ \chi ( C, E ) - h^0 ( C, F ) \ = \ d - \delta_{r+1}(C) + g - 1 - h^1 ( C, F ) . \]
In particular, for $d \gg 0$, the subspace $V$ is not general in $\Gr ( r + 2 , H^0 ( C, E ) )$.
\item Suppose that $r = 2$. Then \cite[Theorem 1.5]{BMNO} implies that if $C$ is general, for certain $d \le 2 \left( \left\lceil \frac{2g}{3} \right\rceil + 2 \right)$ the only $(E, V)$ with nonstable $\mev$ are those containing \emph{pencils}; that is subsystems of type $(1, t, 2)$. However, the above does not contradict this, as for us, when $C$ is general, $d \ge 2 \left( \left\lceil \frac{3g}{4} + 3 \right\rceil \right)$.
\end{enumerate}
\end{remark}

\section{A system of type \texorpdfstring{$(r, d, r+2)$}{(r, d, r+2)} with stable kernel bundle} \label{rdr+2}

In this section, we give an example similar to that in \cite[{\S} 5.3.1]{CHL}, in which linear stability of $(E, V)$ does imply slope stability of $\mev$.

\subsection{A sufficient condition for stability of the DSB}

Here we give a generalisation of \cite[Lemma 5.10]{CHL}, whose proof is very similar.

\begin{lemma} \label{LinStImpliesSlopeSt}
Suppose that $r \ge 2$ and $d < 2 \cdot \delta_{r+1} (C)$. Let $(E, V)$ be a generated coherent system of type $(r, d, r+2)$ over $C$ where $E$ is an ample bundle. If $(E, V)$ is linearly stable, then $\mev$ is slope stable.
\end{lemma}

\begin{proof}
Suppose that $\mev$ is not stable. Let $S \subset \mev$ be a destabilising line subbundle. Following \cite{CT}, we define a subspace $W \subseteq V$ by setting
\begin{equation} \label{DefnW}
W^\vee \ := \ \Image \left( V^\vee \ \to \ H^0 ( C, \mevd ) \ \to \ H^0 ( C, S^\vee ) \right) ,
\end{equation}
and form the Butler diagram of $(E, V)$ by $S$:
\[ \xymatrix{
 & 0 \ar[d] & 0 \ar[d] \ar[r] & N \ar[d] \ar[r] & \cdots \\
0 \ar[r] & S \ar[r] \ar[d] & \Oc \otimes W \ar[r] \ar[d] & F_S \ar[r] \ar[d] & 0 \\
0 \ar[r] & \mev \ar[r] \ar[d] & \Oc \otimes V \ar[r] \ar[d] & E \ar[r] \ar[d] & 0 \\
\cdots \ar[r] & Q \ar[r] \ar[d] & \Oc \otimes \left( \frac{V}{W} \right) \ar[r] \ar[d] & T \ar[r] \ar[d] & 0 \\
 & 0 & 0 & 0 &
} \]
Firstly, suppose that $N = 0$. Then $F_S$ is precisely the subsheaf of $E$ generated by $W$. Clearly $\rank F_S = \dim W - 1$. Being generated, $F_S$ has nonnegative degree. It is easy to see that the degree is zero only if $S \cong \Oc$, which is impossible since $h^0 (C, S) \le h^0 (C, \mev) = 0$. Thus $\deg F_S > 0$. Now
\[
\lambda ( F_S , W ) \ = \ \frac{\deg F_S}{\dim W - \rank F_S} \ = \ \deg S^\vee .
\]
As $S$ destabilises $\mev$, we have $\deg S^\vee \le \mu ( \mevd) \ = \ \frac{d}{2}$. Therefore,
\[
\lambda ( F_S , W ) \ \le \ \frac{d}{2} \ = \ \frac{d}{(r+2) - r} \ = \ \lambda (E, V) .
\]
Moreover, since $E$ is ample, $T = E/F_S$ is not trivial. Therefore, $(F_S, W)$ linearly destabilises $(E, V)$.

On the other hand, suppose that $N$ is nonzero. By the snake lemma, we have an exact sequence
\[
0 \ \to \ N \ \to \ Q \ \to \Oc \otimes \left( \frac{V}{W} \right) \ \to \ T \ \to \ 0 .
\]
Thus $N$ injects into $Q$ and $Q/N$ injects into a trivial sheaf. As $Q$ has rank one and $\Oc$ is torsion free, $Q \cong N$ and $\Oc \otimes \left( \frac{V}{W} \right) \cong T$. Due to ampleness of $E$, we have $T = 0$, and so $W = V$. By (\ref{DefnW}), this implies in particular that $h^0 ( C, S^\vee ) \ge r+2$, whence $\deg S^\vee \ge \delta_{r+1} (C)$. As $S$ destabilises $\mev$, we have
\[
\delta_{r+1} \ \le \ \deg ( S^\vee ) \ \le \ \frac{d}{2} ,
\]
and $d \ge 2 \cdot \delta_{r+1} (C)$.

In summary, we have shown that if $\mev$ is not stable, then $(E, V)$ is not linearly stable and/or $d \ge 2 \cdot \delta_{r+1} (C)$. The statement follows.
\end{proof}

\subsection{A system with stable DSB}

We proceed in steps with the construction of the desired coherent system. For the remainder of the section, we suppose that $r \ge 2$ and $g \ge \max\{ (r-1)(r+2) , 6 \}$, and fix a Brill--Noether general curve $C$ of genus $g$. To ease notation, we abbreviate $\delta_i (C)$ to $\delta_i$ for $i \ge 1$. Choose $e$ such that
\begin{equation} \label{NumHyp}
\delta_{r-1} + 2 \ \le \ e \ \le \delta_{r+1} . 
\end{equation}

\noindent \underline{Step 1.} We show that there exists a bundle $F \to C$ of rank $r - 1$ and degree $e$ which is ample and generated with $h^0 (C, F) = r$:

Following \cite{BBPN}, we denote by $G_L (r-1, e, r)$ the moduli space of $\alpha$-stable coherent systems of type $(r-1, e, r)$ where $\alpha \gg 0$. We write $B (r-1, e, r)$ for the Brill--Noether variety parametrising stable bundles of rank $r-1$ and degree $e$ over $C$ with at least $r$ independent sections. One checks that the numerical hypotheses on $e$, $g$ and $r$ imply that
\[
g + ( r - 1 ) - \left\lceil \frac{g}{r} \right\rceil \ \le \ e \ \le \ g + ( r - 1 ) .
\]
 Then by \cite[Corollary 9.2 (1)]{BBPN}, the locus $B (r-1, e, r)$ is irreducible of dimension
\[
\beta(r-1, e, r) \ := \ (r-1)^2 (g-1) + 1 - r ( r - e + (r-1)(g-1) ) \ = \
 re - (r-1)(g + r)
\]
and smooth outside $B(r-1, e, r+1)$. 
 By \cite[Corollary 9.2 (2)]{BBPN}, the association $(F, V) \mapsto F$ defines a dominant map $G_L (r-1, e, r) \dashrightarrow B (r-1, e, r)$. Therefore, for general $F \in B (r-1, e, r)$ the coherent system $(F, H^0 (C, F))$ is $\alpha$-stable of type $(r-1, e, r)$ for $\alpha \gg 0$. By \cite[Theorem 3.1 (3)]{BBPN}, after deforming $F$ if necessary, we may assume that $F$ is generated and $h^0 (C, F^\vee ) = 0$. This implies also that $F$ has no quotients of nonpositive degree. As our base field is $\C$, it follows by \cite[Theorem 2.4]{Har71} that $F$ is ample.\\
\\
\noindent \underline{Step 2.} We show that there exists a generated $L \in \Pic^{e-1} (C)$ with $h^0 (C, L) = 2$:

As $e \ge \delta_{r-1} + 2 > \delta_1$, there exists $L$ of degree $e - 1$ with $h^0 (C, L) \ge 2$. A computation shows that the condition $g \ge (r-1)(r+2)$ implies that $e - 1 \le g + 1$. 
 Therefore, as $C$ is Brill--Noether general, after deforming $L$ if necessary we may assume that $h^0 (C, L) = 2$, and that $L$ is generated.\\
\\
\noindent \underline{Step 3.} We show that there exists a nontrivial extension $0 \to F \to E \to L \to 0$ such that every section of $L$ lifts to $E$, and which is ample and generated:

We use a similar approach to that in \cite[Lemma 4.1]{CH}. The desired extension will exist if and only if the cup product map
\begin{equation} \label{CupPr}
H^1 (C, \Hom(L, F)) \ \to \ \Hom \left( H^0 (C, L), H^1 (C, F) \right)
\end{equation}
has nonzero kernel. This is certainly true if $h^1 (C, F) > h^0 (C, L) \cdot h^1 (C, F)$. Let us verify this inequality. By Riemann--Roch, $h^1 (C, \Hom (L, F) ) \ge (r-2)(e-1) - 1 + (r-1) (g-1)$ 
 and
\[
h^1 (C, F) \ = \ h^0 (C, F) - \chi(C, F) \ = \ r - (e - (r-1)(g-1)) \ = \ r - e + (r-1)(g-1) .
\]
The condition $h^1 (C, F) > h^0 (C, L) \cdot h^1 (C, F)$ is therefore
\[
(r-2)(e-1) - 1 + (r-1) (g-1) \ > \ 2 \cdot ( r - e + (r-1)(g-1) ) \ = \ 2r - 2e + 2(r-1)(g-1) ,
\]
which reduces to $e > \frac{r-1}{r} \cdot g + 2$.  
 The latter is satisfied because $r \ge 2$ and
\[
e \ \ge \ \delta_{r-1} + 2 \ \ge \ \frac{r-1}{r} \cdot g + (r - 1) + 2 .
\]
Thus (\ref{CupPr}) has nonzero kernel, and there exists a nontrivial extension $0 \to F \to E \to L \to 0$ such that every section of $L$ lifts to $E$. As both $F$ and $L$ are ample and generated, so is $E$.\\
\\
\noindent \underline{Step 4.} We show that the type $(r, 2e-1, r+2)$ coherent system $(E, H^0 (C, E))$ is linearly stable:

Let $G \subset E$ be a proper generated elementary transformation of positive degree. Then necessarily $h^0 (C, G) = r+1$. There is a diagram
\[ \xymatrix{
0 \ar[r] & F \ar[r] & E \ar[r] & L \ar[r] & 0 \\
0 \ar[r] & G_1 \ar[r] \ar[u] & G \ar[r] \ar[u] & G_2 \ar[r] \ar[u] & 0 .
} \]
As $G_2$ is generated and $G_2 \subseteq L$, we have $1 \le h^0 (C, G_2) \le 2$.

If $h^0 (C, G_2) = 1$ then $G_2 = \Oc$ and $G = G_1 \oplus \Oc$. As $h^0 (C, G) = r+1$, we must have $h^0 (C, G_1) = r$. As $F$ is generated, necessarily $G_1 = F$. By Corollary \ref{LinSlopeEquality}, we have
\[
\lambda (G) \ = \ \lambda (F \oplus \Oc) \ = \ \lambda (F) \ = \ \frac{\deg F}{(r+1) - r} \ = \ e \ > \frac{2e - 1}{2} \ = \ \lambda (E) .
\]

On the other hand, if $h^0 (C, G_2) = 2$ then $G_2 = L$ since $L$ is generated. Then $h^0 (C, G_1) = r - 1$. As $F$ is ample and generated with $h^0 (C, F) = \rank F + 1$.

Now no proper subbundle $F' \subset F$ has $h^0 (C, F') > \rank F$. For: if such an $F'$ existed then $F/F'$ would be a generated bundle with at most $\rank (F/F')$ sections, which must be trivial. But since $F$ is ample, it has no nonzero trivial quotient.

In particular, the image of $\Oc \otimes H^0 (C, G_1) \to F$ is not contained in any proper subbundle of $F$. It follows that $\Oc \otimes H^0 (C, G_1) \to G_1$ is a sheaf injection, with image a rank $r-1$ trivial subsheaf of $G_1$.

Now if $\deg G_1$ were zero, then since $h^0 (C, G_1) = r-1$ we would have $G_1 = \Oc^{\oplus (r-1)}$. Then by Lemma \ref{genGTiB} below, the extension $0 \to \Oc^{\oplus (r-1)} \to G \to L \to 0$ would be trivial. But then $L$ would be a subsheaf of $E$, contradicting the fact that $E$ is a nontrivial extension. We conclude that $\deg G_1 \ge 1$, and therefore that
\[
\lambda (G) \ \ge \ \frac{1 + (e - 1)}{(r+1) - r} \ = \ e \ > \ \frac{2e - 1}{2} \ = \ \lambda(E) .
\]
By Proposition \ref{LinSstCrit}, implication $(2) \Rightarrow (1)$ and Theorem \ref{EquivDefns}, we conclude that $(E, H^0 (C, E))$ is linearly stable.\\
\\
\underline{Step 5.} We can now conclude: By (\ref{NumHyp}), we have $\deg E = 2e - 1 \le 2 \cdot \delta_{r+1} - 1$. By Lemma \ref{LinStImpliesSlopeSt} the kernel bundle $M_E$ is stable, and our example is complete.

\begin{lemma} \label{genGTiB}
Suppose that $C$ is generic and $L$ a line bundle with $h^0 (C, L) = s$ which is not of the form $\overline{L}^2$ for any $\overline{L} \in W^1_{\frac{g+2}{2}}$. Let $\Lambda$ be a vector space of dimension $r-1$ and $0 \to \Oc \otimes \Lambda \to G \to L \to 0$ an extension. If every section of $L$ lifts to $G$, then $G$ is a trivial extension.
\end{lemma}

\begin{proof}
For $r = 2$, this is equivalent to \cite[Proposition 2.2]{GT}, as the cup product map
\[
H^1 (C, \Hom (L, \Oc) ) \ \to \ \Hom \left( H^0 (C, L), H^1 (C, \Oc) \right)
\]
is dual to the Petri multiplication map $H^0 (C, \Kc ) \otimes H^0 (C, L) \to H^0 (C, \Kc L )$.

Suppose that $r \ge 3$, and let $G$ be an extension as above to which every section of $L$ lifts. Choose a splitting $\Lambda = \Lambda_1 \oplus \Lambda_{r-2}$, where $\dim \Lambda_i = i$. We find a diagram
\[ \xymatrix{
\Oc \otimes \Lambda_{r-2} \ar[r] \ar[d]^= & \Oc \otimes \Lambda \ar[r] \ar[d] & \Oc \otimes \Lambda_1 \ar[d] \\
\Oc \otimes \Lambda_{r-2} \ar[r] & G \ar[r] \ar[d] & \bG \ar[d] \\
& L \ar[r]^= & L
} \]
where $\bG$ is a bundle of rank two satisfying $h^0 (C, \bG) = h^0 (C, G) - (r-2) = s + 1$. In particular, every section of $L$ lifts to $\bG$. Thus by \cite[Proposition 2.2]{GT}, there is a splitting $\bG \to \Oc \otimes \Lambda_1$. Thus the composed map $G \to \bG \to \Oc \otimes \Lambda_1$ 
 defines a splitting $G \cong (\Oc \otimes \Lambda_1 ) \oplus \tG$, where $\tG$ is a bundle of rank $r-1$ with $h^0 (C, \tG) = s + r - 2$. Using the snake lemma, we now find a diagram
\[ \xymatrix{
 & & \Oc \otimes \Lambda_{r-2} \ar[d] \\
\Oc \otimes \Lambda_1 \ar[r] \ar[d] & \Oc \otimes \Lambda_1 \oplus \tG \ar[r] \ar[d]^= & \tG \ar[d] \\
\Oc \otimes \Lambda \ar[r] \ar[d] & G \ar[r] & L \\
\Oc \otimes \Lambda_{r-2} & &
} \]
By induction, we may assume that $\tG \cong (\Oc \otimes \Lambda_{r-2} ) \oplus L$. As $G \cong ( \Oc \otimes \Lambda_1 ) \oplus \tG$, we see that $G \cong ( \Oc \otimes \Lambda ) \oplus L$, as desired.
\end{proof}

\begin{remark}
For $r = 2$, the bundle $E$ is desemistabilised by the subbundle $F$. In fact the above construction specialises to \cite[Example 5.1]{CH}, whence we see that the coherent system $(E, H^0 (C, E))$ is also $\alpha$-nonsemistable for $\frac{1}{\alpha} \gg 0$. For $r \ge 3$, the subbundle $F$ does not necessarily destabilise $E$, 
 but we are unable to decide the stability of the bundle $E$ itself with the methods presently at our disposal. However, we hope that the example in this section will nonetheless illustrate the role that linear stability can play in proving stability of kernel bundles of low degree.
\end{remark}

\end{document}